\documentclass[10pt,twocolumn,letterpaper]{article}

\usepackage{iccv}
\usepackage{times}
\usepackage{epsfig}
\usepackage{graphicx}
\usepackage{amsmath}
\usepackage{amssymb}
\usepackage{amsthm}
\usepackage{booktabs}
\usepackage{subfig}
\usepackage{tikz} 
\usepackage{pgfplots}
\usepackage{multirow}
\usepackage{bm}
\usepackage{dsfont}
\usepackage{array}

\DeclareMathOperator*{\argmin}{arg\,min}
\DeclareMathOperator*{\argmax}{arg\,max}

\def\sl{\underline{s}}
\def\su{\overline{s}}
\def\I{{\mathbb{I}}}

\def\x{{\bf x}}

\def\z{{\bf z}}
\def\bb{{\bf b}}

\def\rank{{\text{ rank}}}
\def\tr{{\text{ tr}}}

\def\diag{{\text{ diag}}}

\def\A{{\mathcal{A}}}

\newcommand{\skal}[1]{\langle #1 \rangle}

\newcommand{\reg}{\ensuremath{\mathcal{R}}}

\newtheorem{theorem}{Theorem}[section]
\newtheorem{lemma}[theorem]{Lemma}

\newtheorem{corollary}[theorem]{Corollary}

\iccvfinalcopy 

\def\iccvPaperID{2} 

\ificcvfinal\fi
\begin{document}


\title{A Non-Convex  Relaxation for Fixed-Rank Approximation}
\author{Carl Olsson$^{1,2}$ \hspace{1cm} Marcus Carlsson$^1$  \hspace{1cm} Erik Bylow$^1$ \\[0.2 cm]
\begin{minipage}[c]{0.4\textwidth}
\centering
${}^1$Centre for Mathematical Sciences\\
Lund University 
\end{minipage}
\begin{minipage}[c]{0.4\textwidth}
\centering
${}^2$Department of Electrical Engineering\\
Chalmers University of Technology 
\end{minipage}
 \\[0.4cm]
}

\maketitle

\begin{abstract}
This paper considers the problem of finding a low rank matrix from observations of linear combinations of its elements.
It is well known that if the problem fulfills a restricted isometry property (RIP), convex relaxations using the nuclear norm typically work well and come with theoretical performance guarantees.
On the other hand these formulations suffer from a shrinking bias that can severely degrade the solution in the presence of noise.

In this theoretical paper we study an alternative non-convex relaxation that in contrast to the nuclear norm does not penalize the leading singular values and thereby avoids this bias. We show that despite its non-convexity the proposed formulation will in many cases have a single local minimizer if a RIP holds. Our numerical tests show that our approach typically converges to a better solution than nuclear norm based alternatives  even in cases when the RIP does not hold.
\end{abstract}

\section{Introduction}
Low rank approximation is an important tool in applications such as
rigid and non rigid structure from motion,  photometric stereo and optical flow \cite{tomasi-kanade-ijcv-1992,bregler-etal-cvpr-2000,yan-pollefeys-pami-2008,garg-etal-cvpr-2013,basri-etal-ijcv-2007,garg-etal-ijcv-2013}. 
The rank of the approximating matrix typically describes the complexity of the solution.
For example, in non-rigid structure from motion the rank measures the number of
basis elements needed to describe the point motions~\cite{bregler-etal-cvpr-2000}.
Under the assumption of Gaussian noise the objective is typically to solve
\begin{equation}
\min_{\rank(X)\leq r} \|X-X_0\|_F^2,
\end{equation}
where $X_0$ is a measurement matrix and $\|\cdot\|_F$ is the Frobenius norm.
The problem can be solved optimally using SVD \cite{eckart-young-1936},
but the strategy is limited to problems where all matrix elements are directly measured.
In this paper we will consider low rank approximation problems where linear combinations of the elements are observed.
We aim to solve problems of the form
\begin{equation}
\min_X \I(\rank(X) \leq r) + \|\A X-\bb\|^2.
\label{eq:orgprobl}
\end{equation}
Here $\I(\rank(X) \leq r)$ is $0$ if $\rank(X) \leq r$ and $\infty$ otherwise.
The linear operator $\A:\mathbb{R}^{m\times n} \rightarrow \mathbb{R}^p$ is assumed to fulfill a restricted isometry property (RIP) \cite{recht-etal-siam-2010}
\begin{equation}
(1-\delta_q)\|X\|_F^2 \leq \|\A X\|^2 \leq (1+\delta_q)\|X\|_F^2,
\label{eq:RIP}
\end{equation}
for all matrices with $\rank(X) \leq q$.
The standard approach for problems of this class is to replace the rank function with the convex nuclear norm $\|X\|_* = \sum_i \sigma_i(X)$ \cite{recht-etal-siam-2010,candes-etal-acm-2011}.
It was first observed that this is the convex envelope of the rank function over the set $\{X; \sigma_1(X) \leq 1\}$ in \cite{fazel-etal-acc-2015}.
Since then a number of generalizations that give performance guarantees for the nuclear norm relaxation have appeared, e.g. \cite{recht-etal-siam-2010,oymak2011simplified,candes-etal-acm-2011,candes2009exact}.
The approach does however suffer from a shrinking bias that can severely degrade the solution in the presence of noise. In contrast to the rank constraint the nuclear norm penalizes both small singular values of $X$, assumed to stem from measurement noise, and large singular values, assumed to make up the true signal, equally.
In some sense the suppression of noise also requires an equal suppression of signal.
Non-convex alternatives have been shown to improve performance \cite{oymak-etal-2015,mohan2010iterative}.

In this paper we will consider the relaxation
\begin{equation}
\min_X \reg_r(X) + \|\A X-\bb\|^2,
\label{eq:relaxation}
\end{equation}
where 
\begin{equation}
\reg_r(X) = \max_Z \sum_{i=r+1}^N z_i^2 - \|X-Z\|^2, 
\label{eq:regterm}
\end{equation}
and $z_i$, $i=1,...,N$ are the singular values of $X$ and $Z$ respectively.
The minimization over $Z$ does not have any closed form solution, however it was shown in \cite{larsson-olsson-ijcv-2016,andersson-etal-ol-2017} how to efficiently evaluate and compute its proximal operator.
Figure ~\ref{fig:Rg_3d} shows a three dimensional illustration of the level sets of the regularizer.
\begin{figure}[htb]
\centering
\includegraphics[width=25mm]{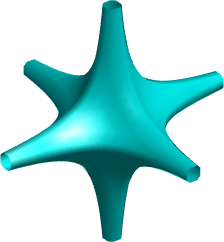}
~
\includegraphics[width=30mm]{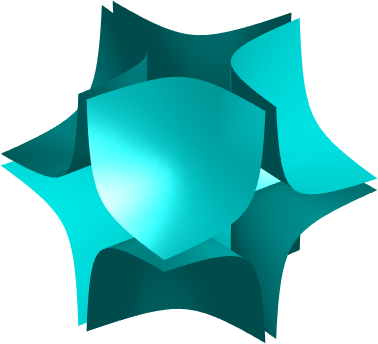}
~
\includegraphics[width=20mm]{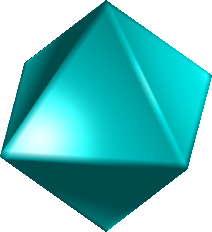}
\caption{Level set surfaces $\{ X ~|~ \reg_{r}(X) = \alpha \}$ for $X = \diag(x_1,x_2,x_3)$ with $r = 1$ (\textit{Left}) and $r = 2$ (\textit{Middle}). Note that when $r = 1$ the regularizer promotes solutions where only one of $x_k$ is non-zero. For $r = 2$ the regularlizer instead favors solutions with two non-zero $x_k$. For comparison we also include the nuclear norm. (\textit{Right})}
\label{fig:Rg_3d}
\end{figure}
In \cite{larsson-olsson-ijcv-2016,larsson-olsson-emmcvpr-2015} it was shown that 
\begin{equation}
\reg_r(X)+ \|X-X_0\|_F^2,
\label{eq:X0relax}
\end{equation}
is the the convex envelope of 
\begin{equation}
\I(\rank(X) \leq r) + \|X-X_0\|_F^2.
\label{eq:X0probl}
\end{equation}
By itself the regularization term $R_r(X)$ is not convex, but when adding a quadratic term $\|X\|_F^2$ the result is convex.
It is shown in \cite{andersson-etal-ol-2017} that \eqref{eq:X0probl} and \eqref{eq:X0relax} have the same optimizer as long as the singular values of $X_0$ are distinct. 
(When this is not the case the minimizer is not unique.)
Assuming that \eqref{eq:RIP} holds $\|\A X\|^2$ will behave roughly like $\|X\|_F$ for matrices of rank less than $q$, and therefore it seems reasonable that \eqref{eq:relaxation} should have some convexity properties.
In this paper we study the stationary points of \eqref{eq:relaxation}.
We show that if a RIP property holds it is in many cases possible to 
guarantee that any stationary point of \eqref{eq:relaxation} (with rank $r$) is unique.

A number of recent works propose to use the related $\|\cdot\|_{r*}$-norm 
\cite{mcdonald-etal-nips-2014,eriksson-etal-cvpr-2015,grussler-etal-arxiv-2016,grussler-rantzer-cdc-2015} (sometimes referred to as the spectral k-support norm). This is a generalization of the nuclear norm which is obtained when selecting $r=1$. It can be shown \cite{mcdonald-etal-nips-2014} that the extreme points of the unit ball with this norm are rank $r$ matrices. Therefore this choice may be more appropriate than the nuclear norm when searching for solutions of a particular (known) rank. It can be seen (e.g. from the derivations in \cite{grussler-etal-arxiv-2016}) that 
$\|X \|_{r*}$ is the convex envelope of \eqref{eq:X0probl} when $X_0=0$, which gives 
$\|X \|^2_{r*} = \reg_r(X) + \|X\|_F^2$. 
While the approach is convex the extra norm penalty adds a (usually) unwanted shrinking bias similar to what the nuclear norm does. In contrast, our approach avoids this bias since it uses a non-convex regularizer. Despite this non-convexity we are still able to derive strong optimality guarantees for an important class of problem instances.

\subsection{Main Results and Contributions}\label{sec:contrib}
Our main result, Theorem \ref{thm:uniqueness}, shows that if $X_s$ is a stationary point of \eqref{eq:relaxation} and the singular values $z_i$ of 
the matrix $Z = (I-\A^* A)X_s+\A^*\bb$ fulfill $z_{r+1} < (1-2\delta_{2r})z_r$ then there can not be any other stationary point with rank less than or equal to $r$. 
The matrix $Z$ is related to the gradient of the objective function at $X_s$ (see Section~\ref{sec:opt}).
The term $\|X-Z\|_F^2$ can be seen as a local approximation of $\|\A X -\bb\|^2$  close to $X_s$ (see \cite{olsson-etal-arxiv-2017}). 
If for example there is a rank $r$ matrix $X_0$ such that $\bb = \A X_0$ then it is easy to show that $X_0$ is a stationary point and the corrresponding $Z$ is identical to $X_0$. Since this means that $z_{r+1} = 0$ our results certify that this is the only stationary point to the problem if $\A$ fulfills \eqref{eq:RIP} with $\delta_{2r} < \frac{1}{2}$.
The following lemma clarifies the connection between the stationary point $X_s$ and $Z$.
\begin{lemma}\label{lemma:lowrankstatpt}
The point $X_s$ is stationary in $F(X)$ = $\reg (X) + \|\A X - \bf{b}\|^2$ if and only if $2Z \in \partial G(X_s)$, where $G(X) = \reg_r(X) + \|X\|_F^2$, and if and only if
\begin{equation}
X_s \in \argmin_X \reg_r(X) + \|X -  Z\|_F^2.
\label{eq:zprobl}
\end{equation}
\end{lemma}
(The proof is identical to that of Lemma 3.1 in \cite{olsson-etal-arxiv-2017}.)
In \cite{andersson-etal-ol-2017} it is shown that as long as $z_r \neq z_{r+1}$ the unique solution of \eqref{eq:zprobl} is the best rank $r$ approximation of $Z$. When there are several singular values that are equal to $z_r$, \eqref{eq:zprobl} will have multiple solutions and some of them will not be of rank $r$.

In \cite{carlsson2016convexification} the relationship between minimizers of \eqref{eq:relaxation} and \eqref{eq:orgprobl} is studied.
We note that \cite{carlsson2016convexification} shows that if $\|\A\| \leq 1$ then any local minimizer of \eqref{eq:relaxation} is also a local minimizer of \eqref{eq:orgprobl}, and that their global minimizers coincide. In this situation local minimizers of $F$ will therefore be rank $r$ approximations of $Z$. Hence, loosely speaking our results state that in this situation any local minimum of $F$ is likely to be unique.

Our work builds on that of \cite{olsson-etal-arxiv-2017} which derives similar results for the non-convex regularizer $\mathcal{R}_\mu(X) = \sum_i \mu - \max(\sqrt{\mu}-x_i)^2$, where $x_i$ are the singular values of $X$. In this case a trade-off between rank and residual error is optimized using the formulation 
\begin{equation}
\min_X \reg_\mu (X) + \|\A X - \bb\|^2.
\label{eq:murelaxation}
\end{equation}
While it can be argued that \eqref{eq:murelaxation} and \eqref{eq:relaxation} are essentially equivalent since we can iteratively search for a $\mu$ that gives the desired rank, the results of \cite{olsson-etal-arxiv-2017} may not rule out the existence of multiple high rank stationary points. 
In contrast, when using \eqref{eq:relaxation} our results imply that if $z_{r+1} < (1-2\delta_r)z_r$ then $X_s$ is the unique stationary point of the problem. (To see this, note that if there are other stationary points we can by the preceding discussion assume that at least one is of rank $r$ or less which contradicts our main result in Theorem~\ref{thm:uniqueness}).
Hence, in this sense our results are stronger than those of \cite{olsson-etal-arxiv-2017} and allow for directly searching for a matrix of the desired rank with an essentially parameter free formulation.

\subsection{Notation}
In this section we introduce some preliminary material and notation. 
In general we will use boldface to denote a vector $\x$ and its $i$th element $x_i$.
Unless otherwise stated the singular values of a matrix $X$ will be denoted $x_i$ and the vector of singular values $\x$.
By $\|\x\|$ we denote the standard euclidean norm $\|\x\| = \sqrt{\x^T \x}$.
A diagonal matrix with diagonal elements $\x$ will be denoted $D_\x$.
For matrices we define the scalar product as $\skal{X,Y} = \tr(X^T Y)$, where $\tr$ is the trace function, and the Frobenius norm $\|X\|_F = \sqrt{\skal{X,X}}=\sqrt{\sum_{i=1}^n x_i}$.
The adjoint of the linear matrix operator $\A$ is denoted $\A^*$.
By $\partial F(X)$ we mean the set of subgradients of the function $F$ at $X$ and by a stationary point we mean a solution to 
$0 \in \partial F(X)$.

\section{Optimality Conditions}\label{sec:opt}
Let $F(X) =  \reg_r(X) + \|\A X-\bb\|^2$. We can equivalently write
\begin{equation}
F(X) = G(X) - \delta_q \|X\|_F^2 + H(X)+\|\bb\|^2,
\end{equation}
where 
$G(X) = \reg_r(X) + \|X\|_F^2$ and $H(X) = \delta_q \|X\|^2_F + \left( \|\A X \|^2 - \|X\|_2^2 \right)-2\skal{ \A X, \bb}$.
The function $G$ is convex and sub-differentiable.
Any stationary point $X_s$ of $F$ therefore has to fulfill
\begin{equation}
2 \delta_q X_s - \nabla H(X_s) \in \partial G(X_s). 
\end{equation}
Computation of the gradient gives the optimality conditions $2Z \in \partial G(X_s)$ where $Z =  (I -\A^* \A)X_s+\A^*\bb$.

\subsection{Subgradients of $G$}
For our analysis we need to determine the subdifferential $\partial G(X)$ of the function $G(X)$.
Let $\x$ be the vector of singular values of $X$ and $X=UD_\x V^T$ be the SVD.
Using Von Neumann's trace theorem it is easy to see \cite{larsson-olsson-ijcv-2016} that the $Z$ that maximizes \eqref{eq:regterm} has to be of the form
$Z = U D_\z V^T$, where $\z$ are singular values.
If we let
\begin{equation}
L(X,Z) = -\sum_{i=1}^r z_i^2 + 2\skal{Z,X},
\label{eq:Lfunk}
\end{equation}
then we have $G(X) = \max_Z L(X,Z)$. 
The function $L$ is linear in $X$ and concave in $Z$. Furthermore for any given $X$ the corresponding maximizers can be restricted to a compact set (because of the dominating quadratic term). By Danskin's Theorem, see \cite{bertsekas-1999}, the subgradients of $G$ are then given by
\begin{equation}
\partial G(X) = \text{convhull}\{\nabla_X L(X,Z), Z \in \mathcal{Z}(X)\},
\end{equation}
where $\mathcal{Z}(X)=\argmax_Z L(X,Z)$. We note that by concavity the maximizing set $\mathcal{Z}(X)$ is convex. Since $\nabla_X L(X,Z) = 2Z$ we get
\begin{equation}
\partial G(X) = 2\argmax_Z L(X,Z).
\end{equation}

To find the set of subgradients we thus need to determine all maximizers of $L$.
Since the maximizing $Z$ has the same $U$ and $V$ as $X$ what remains is to determine the singular values of $Z$.
It can be shown \cite{larsson-olsson-ijcv-2016} that these have the form
\begin{equation}
z_i \in 
\begin{cases}
\max(x_i,s) & i \leq r \\
s & i \geq r, \ x_i \neq 0 \\
[0,s] & i > r, \ x_i = 0
\end{cases}.
\label{eq:subgrad}
\end{equation}
for some number $s \geq x_r$. (The case $x_i = 0$, $i>r$ is actually not addressed in \cite{larsson-olsson-ijcv-2016}. However, it is easy to see that 
any value in $[0,s]$ works since $z_i$ vanishes from \eqref{eq:Lfunk} when $x_i = 0$, $i > r$. In fact, any value $[-s,s]$ works, but we use the convention that singular values are positive. Note that the columns of $U$ that correspond to zero singular values of $X$ are not uniquely defined. We can always achieve a decreasing sequence with $z_i \in [0,s]$ by changing signs and switching order.)

For a general matrix $X$ the value of $s$ can not be determined analytically but has to be computed numerically by 
maximizing a one dimensional concave and differentiable function \cite{larsson-olsson-ijcv-2016}.
If $\rank(X)\leq r$ it is however clear that the optimal choice is $s=x_r$.
To see this we note that since the optimal $Z$ is of the form $UD_\z V^T$ we have
\begin{equation}
L(X,Z) = -\sum_{i=1}^r (z_i-x_i)^2 + \sum_{i=1}^r x_i^2, 
\label{eq:L1}
\end{equation}
if $\rank(x) \leq r$.
Selecting $s=x_r$ and inserting \eqref{eq:subgrad} into \eqref{eq:L1} gives $L(X,Z) = \sum_{i=1}^r x_i^2$, which is clearly the maximum. 
Hence if $\rank(X) = r$ we conclude that the subgradients of $g$ are given by
$2Z = 2U D_\z V^T$ where
\begin{equation}
z_i \in 
\begin{cases}
x_i & i \leq r \\
[0,x_r] & i \geq r
\end{cases}.
\label{eq:subgrad2}
\end{equation}

\subsection{Growth estimates for the $\partial G(X)$}

Next we derive a bound on the growth of the subgradients that will be useful when considering the uniqueness of low rank stationary points. 

Let $\x$ and $\x'$ be two vectors both with at most $r$ non-zero (positive) elements,
and $I$ and $I'$ be the indexes of the $r$ largest elements of $\x$ and $\x'$ respectively.
We will assume that both $I$ and $I'$ contain $r$ elements. 
If in particular $\x' $ has fewer than $r$ non-zero elements we also include some zero elements in $I'$.
We define the corresponding sequences $\z$ and $\z'$ by
\begin{equation}
z_i \in 
\begin{cases}
x_i & i \in I\\
[0,\su] & i \notin I
\end{cases}, \quad
z'_i \in 
\begin{cases}
x'_i & i \in I'\\
[0,\su'] & i \notin I'
\end{cases},
\end{equation}
where $\su = \min_{i\in I} x_i$ and $\su' = \min_{i\in I} x'_i$.
If $\x'$ has fewer than $r$ non-zero elements then $\su'=0$.
Note that we do not require that the elements of the $\x,\x',\z$ and $\z'$ vectors are ordered in decreasing order.
We will see later (Lemma \ref{UVlemma}) that in order to estimate the effects of the $U$ and $V$ matrices we need to be able to handle permutations of the singular values.
For our analysis we will also use the quantity $\sl = \max_{i\notin I} z_i$.

\begin{lemma}\label{lemma:subgradbnd}
If $\sl < c\su$, where $0< c< 1$ then 
\begin{equation}
\skal{\z'-\z,\x'-\x} > \frac{1-c}{2}\|\x'-\x\|^2.
\end{equation}
\end{lemma}
\begin{proof}
Since $z_i = x_i$ when $i\in I$ and $x_i=0$ otherwise, we can write the inner product $\skal{\z'-\z,\x'-\x}$ as
\begin{equation}
\sum_{\footnotesize
\begin{array}{c}
i \in I \\
i \in I'\end{array}}
(x_i-x_i')^2 + 
\sum_{\footnotesize
\begin{array}{c}
i \in I \\
i \notin I'\end{array}}
x_i(x_i-z_i') + 
\sum_{\footnotesize
\begin{array}{c}
i \notin I \\
i \in I'\end{array}}
x'_i(x'_i-z_i).
\label{eq:sum1}
\end{equation}
Note that $\|\x'-\x\|^2=$
\begin{equation}
\sum_{\footnotesize
\begin{array}{c}
i \in I \\
i \in I'\end{array}}
(x_i-x_i')^2 + 
\sum_{\footnotesize
\begin{array}{c}
i \in I \\
i \notin I'\end{array}}
x_i^2 + 
\sum_{\footnotesize
\begin{array}{c}
i \notin I \\
i \in I'\end{array}}
x_{i}^{\prime 2}.
\end{equation}
Since the second and third sum in \eqref{eq:sum1} have the same number of terms it suffices to show that
\begin{equation}
x_i(x_i-z_i') + x_j'(x'_j-z_j) \geq \frac{1-c}{2}(x_i^2 + x_j^{\prime 2}),
\label{eq:twoterms}
\end{equation}
when $i\in I$, $i\notin I'$ and $j \notin I$, $j \in I'$.
By the assumption $\sl < c \su$ we know that $z_j < c x_i$. 
We further know that $z'_i \leq \su' \leq x'_j$. This gives
\begin{eqnarray}
x_i z'_i \leq x_i x_j' \leq \frac{x_i^2 + x_j^{\prime 2}}{2}, \\
x'_j z_j < c x'_j x_i \leq c \frac{x_i^2 + x_j^{\prime 2}}{2},
\end{eqnarray}
Inserting these inequalities into the left hand side of \eqref{eq:twoterms} gives the desired bound.
\end{proof}

The above result gives an estimate of the growth of the subdifferential in terms of the singular values. To derive a similar estimate for the matrix elements we need the following lemma: 

\begin{lemma}\label{UVlemma}
Let $\x$,$\x'$,$\z$,$\z'$ be fixed vectors with non-increasing and non-negative elements 
such that $\x \neq \x'$ and $\z$ and $\z'$ fulfill \eqref{eq:subgrad} (with $\x$ and $\x'$ respectively).
Define $X' = U' D_{\x'} V'^T$, $X = U D_\x V^T$, $Z' = U' D_{\z'} V'^T$, and  $Z = U D_\z V^T$ as functions of unknown orthogonal matrices $U$, $V$, $U'$ and $V'$. 
If
\begin{equation}
a^* = \min_{U,V,U',V'}\frac{\skal{Z'-Z,X'-X}}{\|X'-X\|_F^2} \leq 1
\label{eq:matrixfraction}
\end{equation}
then
\begin{equation}
a^* = \min _{M_\pi} \frac{\skal{M_\pi \z' - \z,M_\pi \x' - \x}}{\|M_\pi \x'-\x\|^2},
\end{equation}
where $M_\pi$ belongs to the set of permutation matrices.
\end{lemma}
The proof is almost identical to that of Lemma 4.1 in \cite{olsson-etal-arxiv-2017} and therefore we omit it.
While our subdifferential is different to the one studied in \cite{olsson-etal-arxiv-2017}, for both of them we have that the $\z$ and $\x$ vectors fulfill $z_i \geq x_i \geq 0$ which is the only property that is used in the proof.

\begin{corollary}\label{cor:subgradestim}
Assume that $X$ is of rank $r$ and $2Z \in \partial G(X)$. 
If the singular values of the matrix $Z$ fulfill $z_{r+1} < cz_{r}$, where $0 < c < 1$, 
then for any $2Z' \in \partial G(X')$ with $\rank(X') \leq r$ we have 
\begin{equation}
\skal{Z'-Z,X'-X} >  \frac{1-c}{2} \|X'-X\|_F^2,
\end{equation}
as long as $\|X'-X\|_F \neq 0$.
\end{corollary}

\begin{proof}
We let $\x,\x',\z$ and $\z'$ be the singular values of the matrices $X,X',Z$ and $Z'$ respectively.
Our proof essentially follows that of Corollary 4.2 in \cite{olsson-etal-arxiv-2017}, where
a similar result is first proven under the assumption that $\x \neq \x'$ and then generalized to the general case using a continuity argument. 
For this purpose we need to extend the infeasible interval somewhat.
Since $0< c < 1$ and $z_{r+1} < c z_r$ are open there is an 
$\epsilon > 0$ such that $z_{r+1} < (c-\epsilon) z_r$ and $0 < c-\epsilon < 1$.
Now assume that $a^* > 1$ in \eqref{eq:matrixfraction}, then clearly
\begin{equation}
\skal{Z'-Z,X'-X} 
> \frac{1-c+\epsilon}{2} \|X'-X\|_F^2,
\label{eq:epsilonbound}
\end{equation}
since $ \frac{1-c+\epsilon}{2}  < 1$. 
Otherwise $a^* \leq 1$ and we have
\begin{equation}
\frac{\skal{Z'-Z,X'-X}}{\|X'-X\|_F^2} \geq  \frac{\skal{M_{\pi} \z' - \z,M_{\pi} \x' - \x}}{\|M_{\pi} \x' - \x\|^2}.
\end{equation}
According to Lemma \ref{lemma:subgradbnd} the right hand side is strictly larger than $ \frac{1-c+\epsilon}{2}$, 
which proves that \eqref{eq:epsilonbound} holds for all $X'$ with $\x' \neq \x$.

It remains to show that 
\begin{equation}
\skal{Z'-Z,X'-X}
\geq  \frac{1-c+\epsilon}{2} \|X'-X\|_F^2,
\label{eq:epsilonbound2}
\end{equation}
for the case $\x' = \x$ and $\|X'-X\|_F \neq 0$.
Since $\epsilon > 0$ is arbitrary this proves the Corollary.
This can be done as in \cite{olsson-etal-arxiv-2017} using continuity of the scalar product and the Frobenius norm. Specifically, a sequence $X(t) \rightarrow X$, when $t \rightarrow 0$, is defined by modifying the largest singular value and letting $\sigma_1(X(t)) = \sigma_1(X) + t$.
It is easy to verify that $X(t)$ fulfills \eqref{eq:epsilonbound} for every $t > 0$. Letting $t \rightarrow 0$ then proves \eqref{eq:epsilonbound2}.
\end{proof}

\subsection{Uniqueness of Low Rank Stationary Points}
In this section we show that if a RIP \eqref{eq:RIP} holds and the singular values $z_r$ and $z_{r+1}$ are well separated there can only be one stationary point of $F$ that has rank $r$.
We first derive a bound on the gradients of $H$.
We have
\begin{equation}
\nabla H(X) = 2 \delta_q X + 2(\A^* \A - I)X-2 \A^* \bb.
\end{equation}
This gives 
$
\skal{\nabla H(X')-\nabla H(X), X'-X} =
$
\begin{equation}
2 \delta_q \|X'-X\|_F^2 + 2(\|\A (X'- X)\|^2-\|X'-X\|_F^2).
\end{equation}
By \eqref{eq:RIP} 
$\left| \|\A (X'- X)\|^2-\|X'-X\|_F^2 \right| \leq \delta \|X'-X\|_F^2 ,
$ if $\rank(X'-X) \leq q$ 
which gives 
\begin{equation}
\skal{\nabla H(X')-\nabla H(X), X'-X} \geq 0
\label{eq:Hest}
\end{equation}
This leads us to our main result
\begin{theorem}\label{thm:uniqueness}
Assume that $X_s$ is a stationary point of $F$, that is, $(I-\A^*\A)X_s + \A^* \bb = Z$,  where $2Z \in \partial G(X_s)$, $\rank(X_s) = r$ and the singular values of $Z$ fulfill $z_{r+1} < (1-2\delta_{2r}) z_r $. If $X'_s$ is another stationary point then $\text{rank}(X'_s) > r$.
\end{theorem}

\begin{proof}
Assume that $\rank(X'_s) \leq r$. Since both $X_s$ and $X'_s$ are stationary we have
\begin{eqnarray}
2 \delta_{2r}X'_s - \nabla H(X'_s) = 2 Z', \\
2 \delta_{2r}X_s - \nabla H(X_s) = 2 Z, 
\end{eqnarray}
where $2Z \in \partial G(X_s)$ and $2Z' \in \partial G(X'_s)$.
Taking the difference between the two equations yields
\begin{equation}
2 \delta_{2r}(X'_s-X_s) - \nabla H(X'_s) + \nabla H(X_s)= 2Z'-2Z,
\end{equation}
which implies
\begin{equation}
2 \delta_{2r}\|V\|_F^2 - \skal{\nabla H(X'_s) + \nabla H(X_s),V}= 2\skal{Z'-Z,V},
\end{equation}
where $V=X'_s-X_s$ has $\rank(V)\leq 2r$.
By \eqref{eq:Hest} the left hand side is less than $2\delta_{2r}\|V\|_F^2$.
However, according to Corollary~\ref{cor:subgradestim} (with $c=1-2\delta_{2r}$) the right hand side is larger than $2\delta_{2r}\|V\|_F^2$ which contradicts $\rank(X'_s) \leq r$.
\end{proof}

\paragraph{Remark.} Note that \cite{carlsson2016convexification} shows that if $\|\A\| \leq 1$ then any local minimizer of \eqref{eq:relaxation} is also a local minimizer of \eqref{eq:orgprobl} and therefore of rank $r$. Hence any local minimizer $X_s$ obeying the conditions of the theorem will be unique.

\section{Implementation and Experiments}

In this section we test the proposed approach on some simple real and synthetic applications (some that fulfill \eqref{eq:RIP} and some that do not). For our implementation we use the GIST approach from \cite{gong-etal-2013} because of its simplicity.
Given a current iterate $X_k$ this method solves
\begin{equation}
X_{k+1} = \argmin_X \reg_r(X) + \tau_k \left\|X- M_k \right\|^2,
\label{eq:tauapprox}
\end{equation}
where $M_k =X_k -\frac{1}{\tau_k}(\A^* \A X_k- A^* \bb) $.
Note that if $\tau_k=1$ then any fixed point of \eqref{eq:tauapprox} is a stationary point by Lemma \ref{lemma:lowrankstatpt}.
To solve \eqref{eq:tauapprox} we use the proximal operator computed in \cite{larsson-olsson-ijcv-2016}.

Our algorithm consists of repeatedly solving \eqref{eq:tauapprox} for a sequence of $\{\tau_k\}$. 
We start from a larger value ($\tau_0 = 5$ in our implementation) and reduce towards $1$ as long as this results in decreasing objective values. Specifically we set $\tau_{k+1} = \frac{\tau_k-1}{1.1} + 1$ if the previous step was successful in reducing the objective value. Otherwise we increase $\tau$ according to $\tau_{k+1} = 1.5(\tau_k-1) + 1$.

\subsection{Synthetic Data}

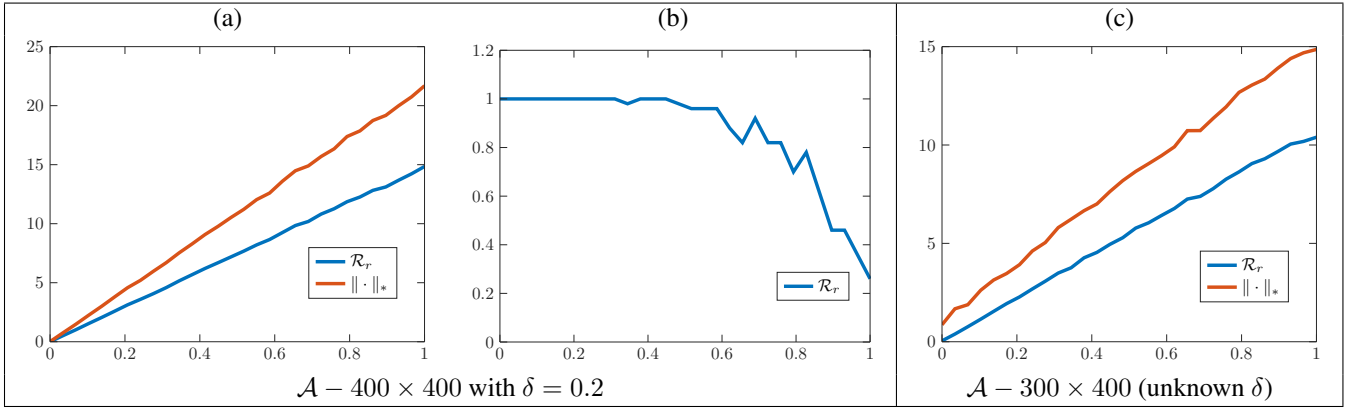
\begin{figure*}
\begin{center}
\def\w{55mm}
\begin{tabular}{|cc|c|}
\hline
(a) & (b) & (c) \\
\resizebox{\w}{!}{
%
%
\definecolor{mycolor1}{rgb}{0.00000,0.44700,0.74100}%
\definecolor{mycolor2}{rgb}{0.85000,0.32500,0.09800}%
\begin{tikzpicture}

\begin{axis}[%
width=76.074mm,
height=60mm,
at={(0mm,0mm)},
scale only axis,
separate axis lines,
every outer x axis line/.append style={white!15!black},
every x tick label/.append style={font=\color{white!15!black}},
every x tick/.append style={white!15!black},
xmin=0,
xmax=1,
xtick={0,0.2,0.4,0.6,0.8,1},
every outer y axis line/.append style={white!15!black},
every y tick label/.append style={font=\color{white!15!black}},
every y tick/.append style={white!15!black},
ymin=0,
ymax=25,
ytick={0,5,10,15,20,25},
axis background/.style={fill=white},
axis on top,
legend style={at={(0.69,0.146)}, anchor=south west, legend cell align=left, align=left, draw=white!15!black}
]
\addplot [color=mycolor1, line width=2.0pt]
  table[row sep=crcr]{%
0	7.81750937991285e-10\\
0.0344827586206897	0.515146930247986\\
0.0689655172413793	1.017437646058\\
0.103448275862069	1.54580047974168\\
0.137931034482759	2.05832443568454\\
0.172413793103448	2.5876262564089\\
0.206896551724138	3.12466210655402\\
0.241379310344828	3.59101887759389\\
0.275862068965517	4.073905826463\\
0.310344827586207	4.58786070200464\\
0.344827586206897	5.15766951333924\\
0.379310344827586	5.68135311581031\\
0.413793103448276	6.20648606494838\\
0.448275862068966	6.68962785351058\\
0.482758620689655	7.18191483417363\\
0.517241379310345	7.6763137048244\\
0.551724137931034	8.20536073091588\\
0.586206896551724	8.65735486315605\\
0.620689655172414	9.25269969482641\\
0.655172413793103	9.84461454783622\\
0.689655172413793	10.187324392088\\
0.724137931034483	10.8177609097617\\
0.758620689655172	11.2660596979874\\
0.793103448275862	11.8608404069798\\
0.827586206896552	12.2603454410226\\
0.862068965517241	12.8230405684574\\
0.896551724137931	13.1125638298326\\
0.931034482758621	13.6848677420051\\
0.96551724137931	14.229574567012\\
1	14.8454120949889\\
};
\addlegendentry{$\mathcal{R}_r$}

\addplot [color=mycolor2, line width=2.0pt]
  table[row sep=crcr]{%
0	0.0109870811192978\\
0.0344827586206897	0.756164946492904\\
0.0689655172413793	1.48654814039559\\
0.103448275862069	2.25540872253526\\
0.137931034482759	3.00627034569975\\
0.172413793103448	3.79817163604483\\
0.206896551724138	4.57596835956441\\
0.241379310344828	5.20117720037665\\
0.275862068965517	5.96921001384714\\
0.310344827586207	6.70098748453494\\
0.344827586206897	7.53309404559638\\
0.379310344827586	8.28686568830322\\
0.413793103448276	9.0796820381963\\
0.448275862068966	9.76814667716797\\
0.482758620689655	10.517398495915\\
0.517241379310345	11.213735215143\\
0.551724137931034	12.0402394464399\\
0.586206896551724	12.5952366744354\\
0.620689655172414	13.6038447255171\\
0.655172413793103	14.4848929386454\\
0.689655172413793	14.8961171783694\\
0.724137931034483	15.6974208158601\\
0.758620689655172	16.3406867238096\\
0.793103448275862	17.3845332471011\\
0.827586206896552	17.8613959166299\\
0.862068965517241	18.7410409251588\\
0.896551724137931	19.1726941073259\\
0.931034482758621	19.9971183870983\\
0.96551724137931	20.7469000156798\\
1	21.6832539706736\\
};
\addlegendentry{$\|\cdot\|_*$}

\end{axis}
\end{tikzpicture}%
}&
\resizebox{\w}{!}{
%
%
\definecolor{mycolor1}{rgb}{0.00000,0.44700,0.74100}%
\begin{tikzpicture}

\begin{axis}[%
width=76.074mm,
height=60mm,
at={(0mm,0mm)},
scale only axis,
separate axis lines,
every outer x axis line/.append style={white!15!black},
every x tick label/.append style={font=\color{white!15!black}},
every x tick/.append style={white!15!black},
xmin=0,
xmax=1,
xtick={0,0.2,0.4,0.6,0.8,1},
every outer y axis line/.append style={white!15!black},
every y tick label/.append style={font=\color{white!15!black}},
every y tick/.append style={white!15!black},
ymin=0,
ymax=1.2,
ytick={0,0.2,0.4,0.6,0.8,1,1.2},
axis background/.style={fill=white},
axis on top,
legend style={at={(0.748,0.142)}, anchor=south west, legend cell align=left, align=left, draw=white!15!black}
]
\addplot [color=mycolor1, line width=2.0pt]
  table[row sep=crcr]{%
0	1\\
0.0344827586206897	1\\
0.0689655172413793	1\\
0.103448275862069	1\\
0.137931034482759	1\\
0.172413793103448	1\\
0.206896551724138	1\\
0.241379310344828	1\\
0.275862068965517	1\\
0.310344827586207	1\\
0.344827586206897	0.980000000000001\\
0.379310344827586	1\\
0.413793103448276	1\\
0.448275862068966	1\\
0.482758620689655	0.980000000000001\\
0.517241379310345	0.960000000000001\\
0.551724137931034	0.960000000000001\\
0.586206896551724	0.960000000000001\\
0.620689655172414	0.88\\
0.655172413793103	0.82\\
0.689655172413793	0.92\\
0.724137931034483	0.82\\
0.758620689655172	0.82\\
0.793103448275862	0.7\\
0.827586206896552	0.78\\
0.862068965517241	0.62\\
0.896551724137931	0.46\\
0.931034482758621	0.46\\
0.96551724137931	0.36\\
1	0.26\\
};
\addlegendentry{$\mathcal{R}_r$}

\end{axis}
\end{tikzpicture}%
}&
\resizebox{\w}{!}{
%
%
\definecolor{mycolor1}{rgb}{0.00000,0.44700,0.74100}%
\definecolor{mycolor2}{rgb}{0.85000,0.32500,0.09800}%
\begin{tikzpicture}

\begin{axis}[%
width=76.074mm,
height=60mm,
at={(0mm,0mm)},
scale only axis,
separate axis lines,
every outer x axis line/.append style={white!15!black},
every x tick label/.append style={font=\color{white!15!black}},
every x tick/.append style={white!15!black},
xmin=0,
xmax=1,
xtick={0,0.2,0.4,0.6,0.8,1},
every outer y axis line/.append style={white!15!black},
every y tick label/.append style={font=\color{white!15!black}},
every y tick/.append style={white!15!black},
ymin=0,
ymax=15,
ytick={0,5,10,15},
axis background/.style={fill=white},
axis on top,
legend style={at={(0.689,0.134)}, anchor=south west, legend cell align=left, align=left, draw=white!15!black}
]
\addplot [color=mycolor1, line width=2.0pt]
  table[row sep=crcr]{%
0	0.0452663477312814\\
0.0344827586206897	0.389690173906979\\
0.0689655172413793	0.763713500930524\\
0.103448275862069	1.1467647712505\\
0.137931034482759	1.54373273834883\\
0.172413793103448	1.93630536268886\\
0.206896551724138	2.2763467050987\\
0.241379310344828	2.6841501110039\\
0.275862068965517	3.07743032489103\\
0.310344827586207	3.49104124142147\\
0.344827586206897	3.75937513403615\\
0.379310344827586	4.26598864804995\\
0.413793103448276	4.54245722713319\\
0.448275862068966	4.94805512525553\\
0.482758620689655	5.28750084159965\\
0.517241379310345	5.77628028894524\\
0.551724137931034	6.04494222002813\\
0.586206896551724	6.41799560793692\\
0.620689655172414	6.7752411220274\\
0.655172413793103	7.25139075452673\\
0.689655172413793	7.38735271193045\\
0.724137931034483	7.78070346063014\\
0.758620689655172	8.2626726281706\\
0.793103448275862	8.63281442311346\\
0.827586206896552	9.05565628142579\\
0.862068965517241	9.30001266828098\\
0.896551724137931	9.66573692656112\\
0.931034482758621	10.0469955543114\\
0.96551724137931	10.1820418805794\\
1	10.3952948669598\\
};
\addlegendentry{$\mathcal{R}_r$}

\addplot [color=mycolor2, line width=2.0pt]
  table[row sep=crcr]{%
0	0.861421534997627\\
0.0344827586206897	1.67399343211803\\
0.0689655172413793	1.88233714359888\\
0.103448275862069	2.61859468194325\\
0.137931034482759	3.13799272877185\\
0.172413793103448	3.46682896876677\\
0.206896551724138	3.9164604072331\\
0.241379310344828	4.61390140320733\\
0.275862068965517	5.04159440053224\\
0.310344827586207	5.8012205665609\\
0.344827586206897	6.22690857864102\\
0.379310344827586	6.66061911182194\\
0.413793103448276	7.00506677782715\\
0.448275862068966	7.63181456614726\\
0.482758620689655	8.18946385855913\\
0.517241379310345	8.65613229404404\\
0.551724137931034	9.05138461750967\\
0.586206896551724	9.45849185990227\\
0.620689655172414	9.90781283255675\\
0.655172413793103	10.7291710598663\\
0.689655172413793	10.732395375938\\
0.724137931034483	11.3534200189758\\
0.758620689655172	11.9368177545908\\
0.793103448275862	12.6746818616958\\
0.827586206896552	13.0454480380315\\
0.862068965517241	13.352408743758\\
0.896551724137931	13.9012469133356\\
0.931034482758621	14.3931870399119\\
0.96551724137931	14.6860404669176\\
1	14.8589286193545\\
};
\addlegendentry{$\| \cdot \|_*$}

\end{axis}
\end{tikzpicture}%
}\\
\multicolumn{2}{|c|}{$\A - 400 \times 400$ with $\delta = 0.2$} & 
$\A- 300 \times 400$ (unknown $\delta$)\\
\hline
\end{tabular}
\end{center}
\caption{(a) - Noise level (x-axis) vs. data fit $\|\A X- \bb\| $ (y-axis) for the solutions obtained with \eqref{eq:relaxation} and \eqref{eq:nuclearrelax}.
(b) - Fraction of instances where the solution of \eqref{eq:relaxation} could be verified to be globally optimal.
(c) - Same as (a). (a) and (b) uses $400 \times 400$ $\A$ with $\delta = 0.2$ while (c)  uses $300 \times 400$ $\A$. }
\label{fig:syntres}
\end{figure*}

We first evaluate the quality of the relaxation on a number of synthetic experiments.
We compare the two formulations \eqref{eq:relaxation} and 
\begin{equation}
\min_X \mu \|X\|_* +\|\A X -\bb\|^2.
\label{eq:nuclearrelax}
\end{equation}

In Figure~\ref{fig:syntres} (a) we tested these two relaxations on a number of synthetic problems with varying noise levels. 
The data was created so that the operator $\A$ fulfills \eqref{eq:RIP} with $\delta=0.2$.
By column stacking an $m\times n$ matrix $X$ the linear mapping $\A$ can be represented with a matrix $A$ of size $p\times mn$. It is easy to see that if we let $p=mn$ the term $(1-\delta_q)$ of \eqref{eq:RIP} will  be the same as the smallest singular value of $A$ squared.
(In this case the RIP constraint will hold for any rank and we therefore suppress the subscript $q$ and only use $\delta$.)
For the data in Figure~\ref{eq:RIP} (a) we selected $400 \times 400$  matrices $A$ with random $\mathcal{N}(0,1)$ (gaussian mean $0$ and variation $1$) entries and modified their singular values. We then generated a $20 \times 20$ matrices $X$ of rank $5$ by sampling $20 \times 5$ matrices $U$ and $V$ with $\mathcal{N}(0,1)$ entries and computed $X = UV^T$. The measurement vector $\bb$ was created by computing $\bb = \A X+{\bm \epsilon}$, where 
${\bm \epsilon}$ is $\mathcal{N}(0,\sigma^2)$ for varying noise level $\sigma$ between $0$ and $1$.
In Figure~\ref{fig:syntres} (a) we plotted the measurement fit $\|\A X -\bb \|$ versus the noise level $\sigma$ for the solutions obtained with \eqref{eq:relaxation} and \eqref{eq:nuclearrelax}.
Note that since the formulation \eqref{eq:nuclearrelax} does not directly specify the rank of the sought matrix we iteratively searched for the smallest value of $\mu$ that gives the correct rank,
using a bisection scheme. The reason for choosing the smallest $\mu$ is that this reduces the shrinking bias to a minimum while it still gives the correct rank.

In Figure~\ref{fig:syntres} (b) we computed the $Z$ matrix and plotted the fraction of problem instances where its singular values fulfilled $z_{r+1} < (1-2 \delta)z_{r}$, with $\delta = 0.2$.
For these instances the obtained stationary points are also globally optimal according to our main results.

In Figure~\ref{fig:syntres} (c) we did the same experiment as in (a) but with an under determined $A$ of size $300 \times 400$. It is known \cite{recht-etal-siam-2010} that if $A$ is $p \times mn$ and the elements of $A$ are drawn from $\mathcal{N}(0,\frac{1}{p})$ then $\A$ fulfills \eqref{eq:RIP} with high probability.
The exact value of $\delta_q$ is however difficult to determine and therefore we are not able to verify optimality in this case.

\subsection{Non-Rigid Structure from Motion}

\begin{figure*}
\begin{center}
\def\w{9mm}
\begin{tabular}{|c|c|c|c|}
\hline
\includegraphics[width=\w]{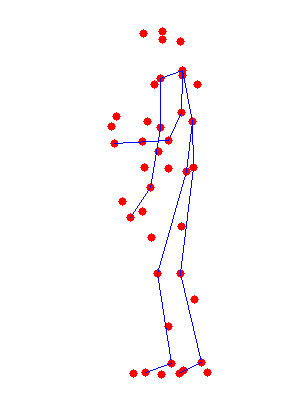}
\includegraphics[width=\w]{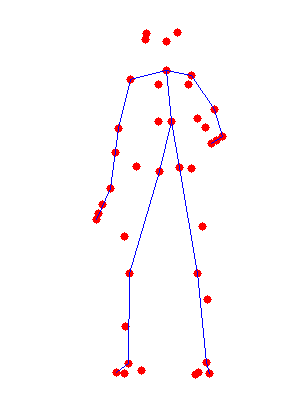}
\includegraphics[width=\w]{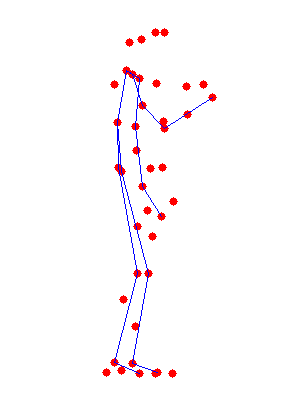}
\includegraphics[width=\w]{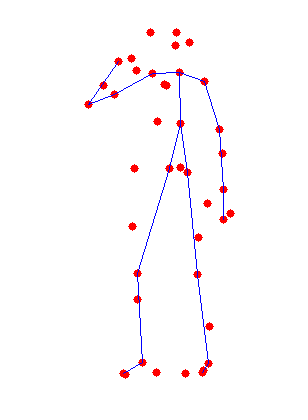}&
\includegraphics[width=\w]{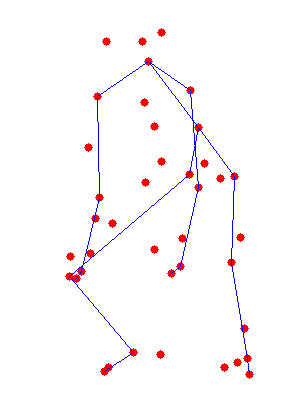}
\includegraphics[width=\w]{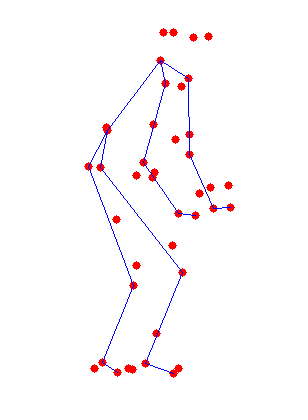}
\includegraphics[width=\w]{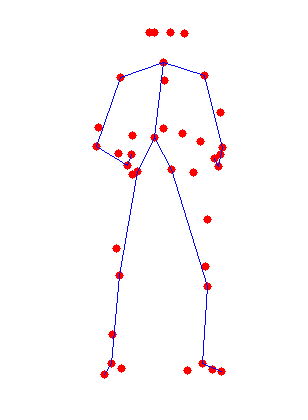}
\includegraphics[width=\w]{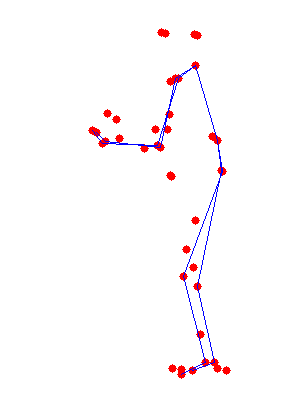}&
\includegraphics[width=\w]{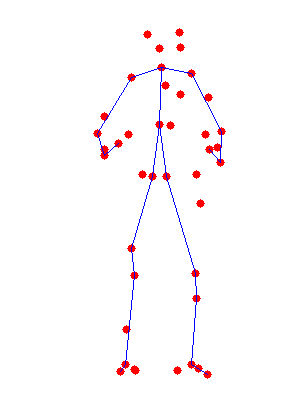}
\includegraphics[width=\w]{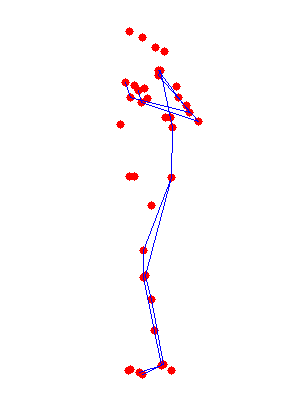}
\includegraphics[width=\w]{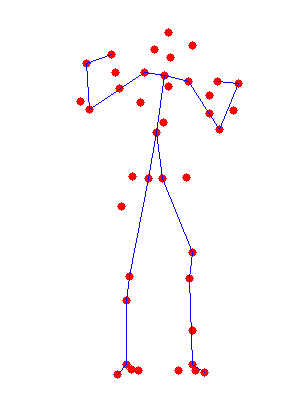}
\includegraphics[width=\w]{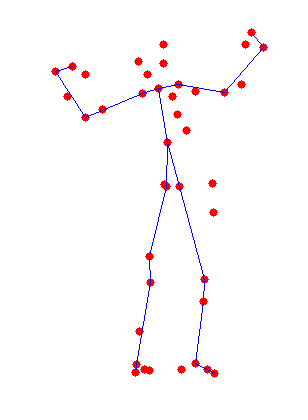}&
\includegraphics[width=\w]{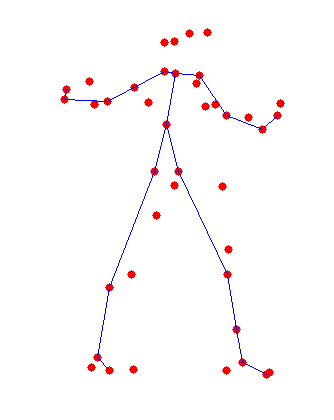}
\includegraphics[width=\w]{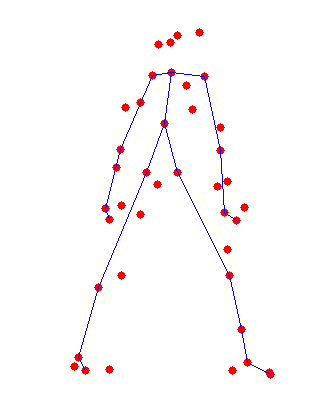}
\includegraphics[width=\w]{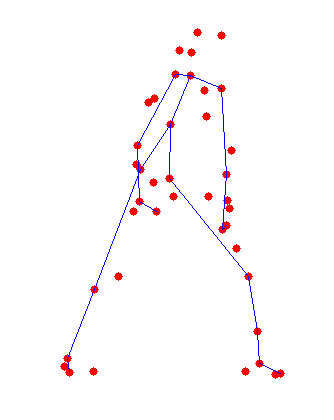}
\includegraphics[width=\w]{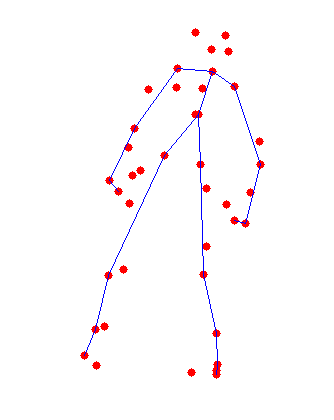}\\
\emph{Drink} & \emph{Pick-up} &
\emph{Stretch} & \emph{Yoga} \\
\hline
\end{tabular}
\end{center}
\caption{Four images from each of the MOCAP data sets.}
\label{fig:mocapshapes}
\begin{center}
\def\w{43mm}
\includegraphics[width=\w]{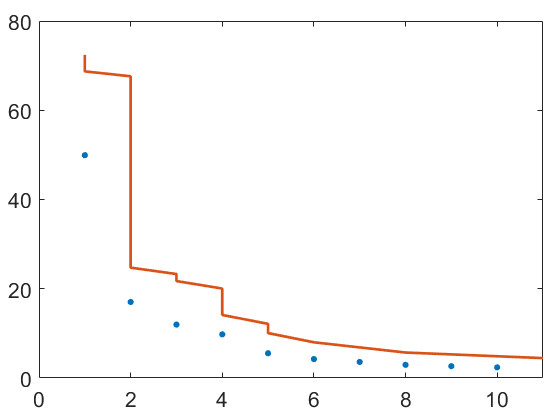}
\includegraphics[width=\w]{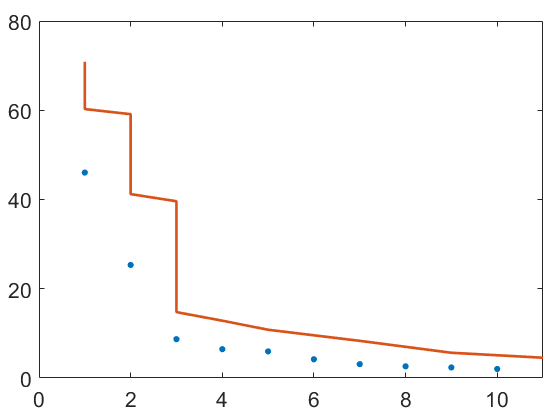}
\includegraphics[width=\w]{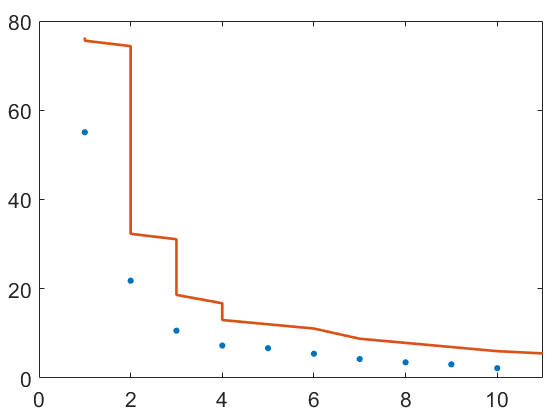}
\includegraphics[width=\w]{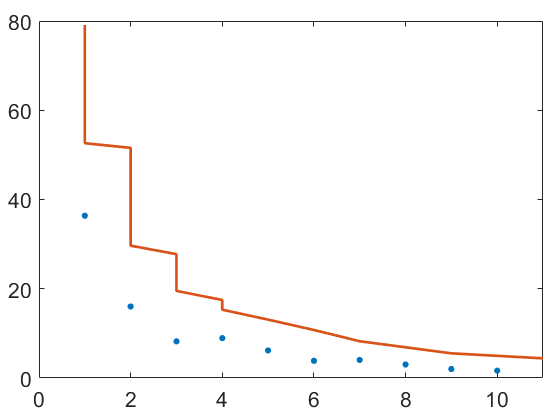}
\end{center}
\caption{Results obtained with \eqref{eq:deformobjrelax} (blue dots) and \eqref{eq:deformobjnuclear} (orage curve) for the four sequences.
Data fit $\|R X - M\|_F$ (y-axis) versus $\rank(X^\#)$ (x-axis).}
\label{fig:deformdatafit}
\begin{center}
\def\w{43mm}
\includegraphics[width=\w]{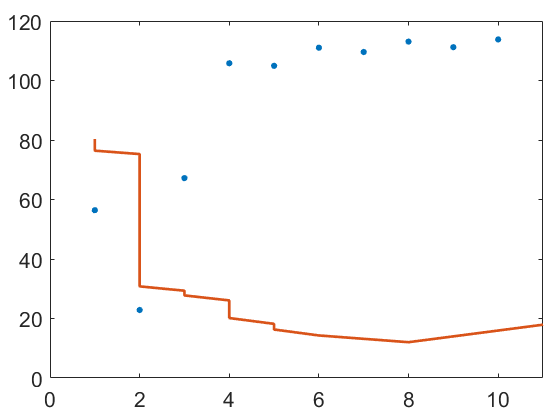}
\includegraphics[width=\w]{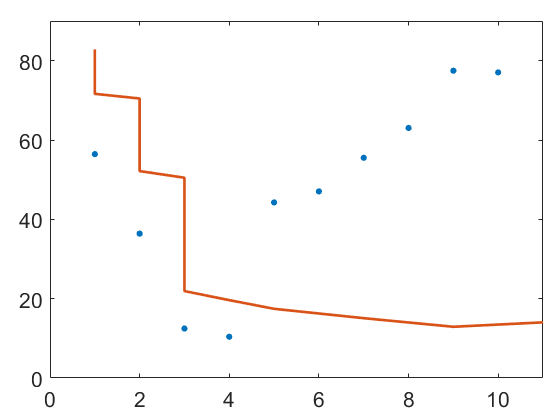}
\includegraphics[width=\w]{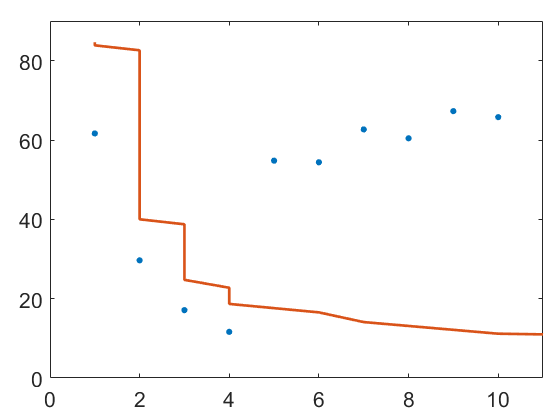}
\includegraphics[width=\w]{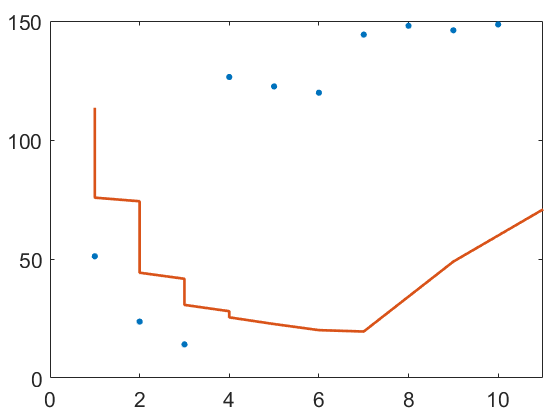}
\end{center}
\caption{Results obtained with \eqref{eq:deformobjrelax} (blue dots) and \eqref{eq:deformobjnuclear} (orage curve) for the four sequences.
Distance to ground truth $\|X - X_{gt}\|_F$ (y-axis) versus $\rank(X^\#)$ (x-axis).}
\label{fig:deformgtdist}
\begin{center}
\def\w{43mm}
\includegraphics[width=\w]{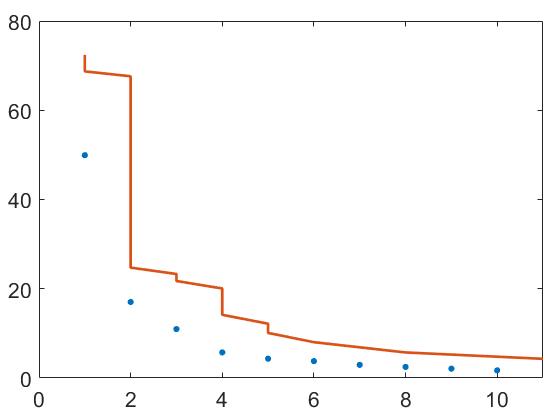}
\includegraphics[width=\w]{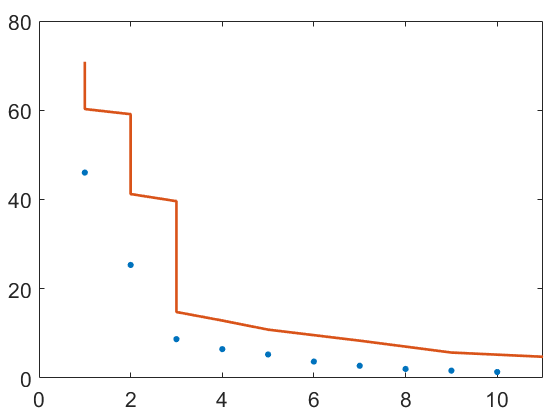}
\includegraphics[width=\w]{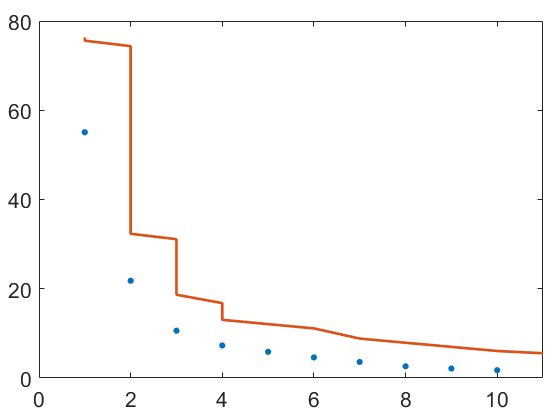}
\includegraphics[width=\w]{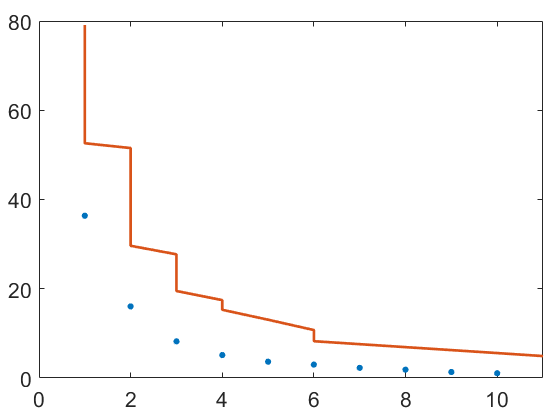}
\end{center}
\caption{Results obtained with \eqref{eq:deformobjrelaxD} (blue dots) and \eqref{eq:deformobjnuclearD} (orage curve) for the four sequences.
Data fit $\|R X - M\|_F$ (y-axis) versus $\rank(X^\#)$ (x-axis).}
\label{fig:deformdatafitD}
\begin{center}
\def\w{43mm}
\includegraphics[width=\w]{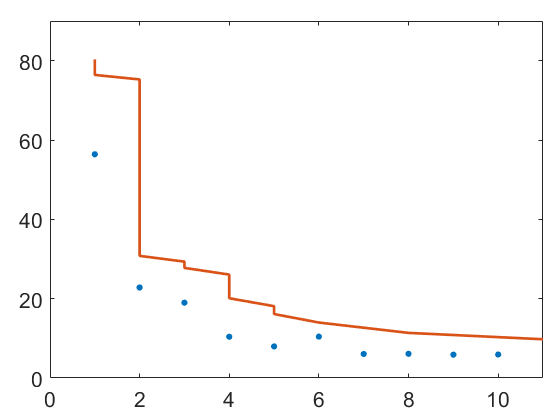}
\includegraphics[width=\w]{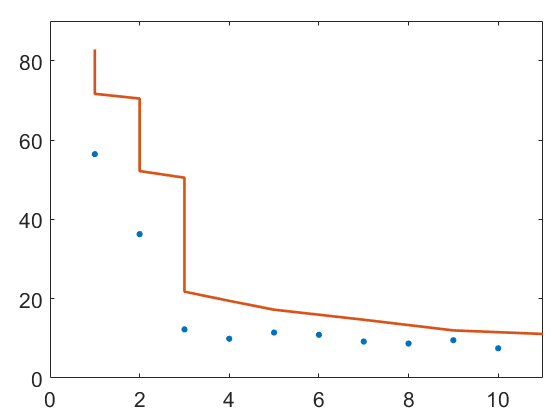}
\includegraphics[width=\w]{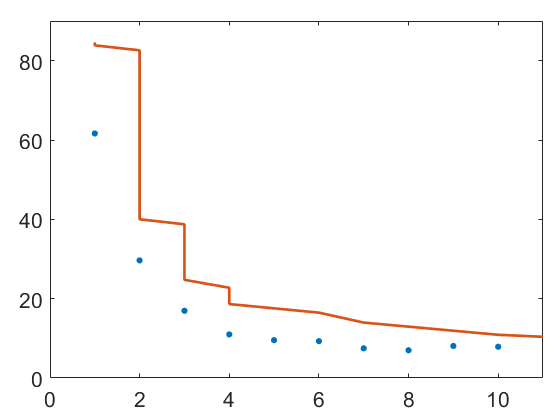}
\includegraphics[width=\w]{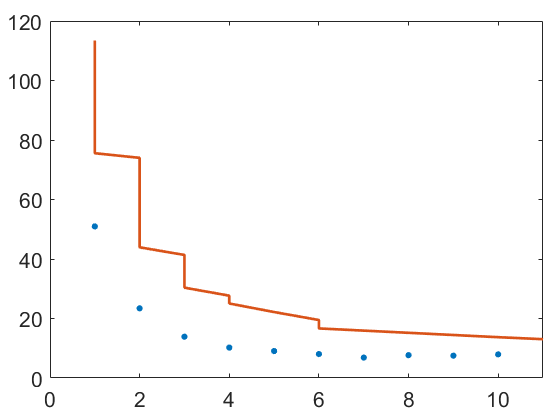}
\end{center}
\caption{Results obtained with \eqref{eq:deformobjrelaxD} (blue dots) and \eqref{eq:deformobjnuclearD} (orage curve) for the four sequences.
Distance to ground truth $\|X - X_{gt}\|_F$ (y-axis) versus $\rank(X^\#)$ (x-axis) is plotted for various regularization strengths.}
\label{fig:deformgtdistD}
\end{figure*}

In this section we consider the problem of Non-Rigid Structure from Motion.
We follow the aproach of Dai. \etal \cite{dai-etal-ijcv-2014} and let
\begin{equation}
X = 
\left[
\begin{array}{c}
X_1 \\
Y_1 \\
Z_1 \\
\vdots \\
X_F \\
Y_F \\
Z_F \\
\end{array}
\right] \text{ and }
X^\# = 
\left[
\begin{array}{ccc}
X_1 & Y_1 & Z_1 \\
\vdots & \vdots & \vdots \\
X_F &
Y_F &
Z_F 
\end{array}
\right],
\end{equation}
where $X_i$,$Y_i$,$Z_i$ are $1 \times m$ matrices containing the $x$-,$y$- and $z$-coordinates of the tracked points in images $i$.
Under the assumption of an orthographic camera the projection of the $3D$ points can be modeled using $M = R X$,
where $R$ is a $2F \times 3F$ block diagonal matrix with $2 \times 3$ blocks $R_i$, consisting of two orthogonal rows that encode the camera orientation in image $i$. The resulting $2F \times m$ measurement matrix $M$ consists of the $x$- and $y$-image coordinates or the tracked points. Under the assumption of a linear shape basis model \cite{bregler-etal-cvpr-2000} with $r$ deformation modes, the matrix $X^\#$ can be factorized into $X^\# = CB$, where the $r \times 3m$ matrix $B$ contain the basis elements. 
It is clear that such a factorization is possible when $X^\#$ is of rank $r$. We therefore search for the matrix $X^\#$ of rank $r$ that minimizes the residual error $\|PX-M\|_F^2$.

The linear operator defined by $\A (X^\#) = RX$ does by itself not obey \eqref{eq:RIP} since there are typically low rank matrices in its nullspace.
This can be seen by noting that if $N_i$ is the $3 \times 1$ vector perpendicular to the two rows of $R_i$, that is $R_i N_i = 0$ then
\begin{equation}
\left[
\begin{array}{c}
X_i \\
Y_i \\
Z_i \\
\end{array}
\right] = N_i C_i,
\end{equation}
where $C_i$ is any $1\times m$ matrix, is in the null space of $R_i$. Therefore any matrix of the form
\begin{equation}
N(C) = \left[  
\begin{array}{ccc}
n_{11} C_1 & n_{21} C_1 & n_{31}C_1  \\
n_{12} C_2 & n_{22} C_2 & n_{32}C_2 \\
\vdots & \vdots & \vdots \\
n_{1F} C_F & n_{2F} C_F & n_{3F}C_F \\
\end{array}
\right],
\end{equation}
where $n_{ij}$ are the elements of $N_i$, vanishes under $\A$.
Setting everything but the first row of $N(C)$ to zero shows that there is a matrix of rank $1$ in the null space of $\A$.
Moreover, if the rows of the optimal $X^\#$ spans such a matrix it will not be unique since we may add $N(C)$ without affecting the projections or the rank.

In Figure~\ref{fig:deformdatafit} we compare the two relaxations
\begin{equation}
\reg_r (X^\#) + \|R X - M\|^2_F
\label{eq:deformobjrelax}
\end{equation}
and
\begin{equation}
\mu \|X^\#\|_* + \|R X - M\|^2_F
\label{eq:deformobjnuclear}
\end{equation}
on the four MOCAP sequences displayed in Figure~\ref{fig:mocapshapes}, obtained from \cite{dai-etal-ijcv-2014}.
These consist of real motion capture data and therefore the ground truth solution is only approximatively of low rank. 

In Figure~\ref{fig:deformdatafit} we plot the rank of the obtained solution versus the datafit $\|R X - M\|^2_F$. 
Since \eqref{eq:deformobjnuclear} does not allow us to directly specify the rank of the sought matrix,
we solved the problem for $50$ values of $\mu$ between $1$ and $100$ (orange curve) and computed the resulting rank and datafit.
Note that even if a change of $\mu$ is not large enough to change the rank of the solution it does affect the non-zero singular values.
To achieve the best result for a specific rank with \eqref{eq:deformobjnuclear} we should select the smallest $\mu$ that gives the correct rank.
Even though \eqref{eq:RIP} does not hold, the relaxation \eqref{eq:deformobjrelax} consistently gives better data fit with lower rank than \eqref{eq:deformobjnuclear}. Figure~\ref{fig:deformgtdist} also shows the rank versus the distance to the ground truth solution.
For high rank the distance is typically larger for \eqref{eq:deformobjrelax} than \eqref{eq:deformobjnuclear}. 
A feasible explanation is that when the rank is high it is more likely that the row space of $X^\#$ contains a matrix of the type $N(C)$.
Loosely speaking, when we allow too complex deformations it becomes more difficult to uniquely recover the shape.
The nuclear norm's built in bias to small solutions helps to regularize the problem when the rank constraint is not discriminative enough.

One way to handle the null space of $\A$ is to add additional regularizes that penalize low rank matrices of the type $N(C)$. 
Dai \etal \cite{dai-etal-ijcv-2014} suggested to use the derivative prior $\|DX^\#\|_F^2$, where the matrix $D:\mathbb{R}^F \rightarrow \mathbb{R}^{F-1}$ is a first order difference operator.
The nullspace of $D$ consists of matrices that are constant in each column. 
Since this implies that the scene is rigid it is clear that $N(C)$ is not in the nullspace of $D$.
We add this term and compare
\begin{equation}
\reg_r(X^\#) + \|R X - M\|^2_F +\|DX^\#\|_F^2
\label{eq:deformobjrelaxD}
\end{equation}
and
\begin{equation}
\mu \|X^\#\|_* + \|R X - M\|^2_F+ \|DX^\#\|_F^2.
\label{eq:deformobjnuclearD}
\end{equation}
Figures~\ref{fig:deformdatafitD} and~\ref{fig:deformgtdistD} show the results. 
In this case both the data fit and the distance to the ground truth is consistently better with \eqref{eq:deformobjrelaxD} than \eqref{eq:deformobjnuclearD}. When the rank increases most of the regularization comes from the derivative prior leading to both methods providing similar results.

\section{Conclusions}
In this paper we studied the local minima of a non-convex rank regularization approach. 
Our main theoretical result shows that if a RIP property holds then there is often a unique local minimum. 
Since the proposed relaxation \eqref{eq:relaxation} and the original objective \eqref{eq:orgprobl} is shown 
to have the same global minimizers if $\|\A\| \leq 1$ in \cite{carlsson2016convexification} this result is also relevant for the original discontinuous problem. Our experimental evaluation shows that the proposed approach often gives
better solutions than standard convex alternatives, even
when the RIP constraint does not hold.

{\small
\bibliographystyle{ieee}

}

\end{document}